\documentclass{amsart}
\usepackage{amsmath}
\usepackage{eurosym}
\usepackage{amssymb}
\usepackage{amsfonts}
\usepackage{xcolor}
\usepackage{hyperref}

\def\eqref#1{\textcolor{black}{(\ref{#1})}}
\hypersetup{
	hidelinks}
\def\equationautorefname~#1\null{Eq.~(#1)\null}

\newtheorem{theorem}{Theorem}
\theoremstyle{plain}

\newtheorem{definition}{Definition}

\newtheorem{proposition}{Proposition}
\newtheorem{remark}{Remark}

\numberwithin{equation}{section}

\begin{document}
\title[Singular minimal surfaces ]{Singular minimal surfaces which are
minimal}
\author{MUHITTIN EVREN AYDIN$^{1}$}
\address{$^{1}$DEPARTMENT OF MATHEMATICS, FACULTY OF SCIENCE, FIRAT
UNIVERSITY, ELAZIG, 23200, TURKEY}
\author{AYLA ERDUR$^{2}$}
\address{$^{2,3}$DEPARTMENT OF MATHEMATICS, FACULTY OF SCIENCE AND ART,
TEKIRDAG NAMIK KEMAL UNIVERSITY, TEKIRDAG 59100, TURKEY}
\author{MAHMUT ERGUT$^{3}$}
\email{meaydin@firat.edu.tr, aerdur@nku.edu.tr, mergut@nku.edu.tr}

\begin{abstract}
In the present paper, we discuss the singular minimal surfaces in a
Euclidean $3-$space $\mathbb{R}^{3}$ which are minimal. In fact, such a
surface is nothing but a plane, a trivial outcome. However, a non-trivial
outcome is obtained when we modify the usual condition of singular
minimality by using a special semi-symmetric metric connection instead of
the Levi-Civita connection on $\mathbb{R}^{3}$. With this new connection, we
prove that, besides planes, the singular minimal surfaces which are minimal
are the generalized cylinders, providing their explicit equations. A trivial
outcome is observed when we use a special semi-symmetric non-metric
connection. Furthermore, our discussion is adapted to the Lorentz-Minkowski
3-space.
\end{abstract}

\maketitle

\section{INTRODUCTION}

Let $\left( \mathbb{R}^{3},\left\langle \cdot ,\cdot \right\rangle \right) $
be a Euclidean $3-$space and $\mathbf{v}$ a fixed unit vector in $\mathbb{R}%
^{3}.$ Let $\mathbf{r}:M^{2}\rightarrow \mathbb{R}_{+}^{3}\left( \mathbf{v}%
\right) $ be a smooth immersion of an oriented compact surface $M^{2}$ into
the halfspace 
\begin{equation*}
\mathbb{R}_{+}^{3}\left( \mathbf{v}\right) :\left\{ p\in \mathbb{R}%
^{3}:\left\langle p,\mathbf{v}\right\rangle >0\right\} .
\end{equation*}%
Denote $H$ and $\mathbf{n}$ the mean curvature and unit normal vector field
on $M^{2}$. Let $\alpha \in \mathbb{R}.$ The potential $\alpha -$energy of $%
\mathbf{r}$ in the direction of $\mathbf{v}$ is defined by \cite{Lopez7}

\begin{equation*}
E\left( \mathbf{r}\right) =\underset{M^{2}}{\int }\left\langle p,\mathbf{v}%
\right\rangle ^{\alpha }dM^{2},
\end{equation*}%
where $dM^{2}$ is the measure on $M^{2}$ with respect to the induced metric
tensor from the Euclidean metric $\left\langle \cdot ,\cdot \right\rangle $
and $p=\mathbf{r}\left( \tilde{p}\right) ,$ $\tilde{p}\in M^{2}.$

Let $\Sigma :M^{2}\times \left( -\theta ,\theta \right) \rightarrow \mathbb{R%
}_{+}^{3}\left( \mathbf{v}\right) $ be a compactly supported variation of $%
\mathbf{r}$ with variaton vector field $\xi .$ The first variation of $E$
becomes 
\begin{equation*}
E^{\prime }\left( 0\right) =-\underset{M^{2}}{\int }\left( 2H\left\langle 
\mathbf{r},\mathbf{v}\right\rangle -\alpha \left\langle \mathbf{n},\mathbf{v}%
\right\rangle \right) \left\langle \xi ,\mathbf{n}\right\rangle ^{\alpha
-1}dM^{2}.
\end{equation*}%
For all compactly supported variations, the immersion $\mathbf{r}$ is a
critical point of $E$ if and only if%
\begin{equation}
2H\left( \tilde{p}\right) =\alpha \frac{\left\langle \mathbf{n}\left( \tilde{%
p}\right) ,\mathbf{v}\right\rangle }{\left\langle \mathbf{r}\left( \tilde{p}%
\right) ,\mathbf{v}\right\rangle },  \label{1}
\end{equation}%
for some point $\tilde{p}\in M^{2}.$

A surface $M^{2}$ is referred to as \textit{singular minimal surface }or $%
\alpha -$\textit{minimal surface }with respect to the vector $\mathbf{v}$,
if holds \autoref{1}, (\cite{Dierkes1,Dierkes2}). In the particular case $\alpha =1$
and $\mathbf{v}=\left( 0,0,1\right) ,$ the surface $M^{2}$ represents 
\textit{two-dimensional analogue of the catenary} which is known as a model
for the surfaces with the lowest gravity center, in other words, one has
minimal potential energy under gravitational forces \cite{Bohme}, \cite%
{Dierkes3}, \cite{Gil}.

A \textit{translation surface} $M^{2}$ in $\mathbb{R}^{3}$ is a surface that
can be written as the sum of two so-called \textit{translating curves} \cite%
{Darboux}. When the translating curves lie in orthogonal planes, up to a
change of coordinates, the surface $M^{2}$ can be locally given in the
explicit form $z=p(x)+q(y),$ where $\left( x,y,z\right) $ is the rectangular
coordinates and $p,q$ smooth functions. In such case, if $M^{2}$ is minimal (%
$H$ vanishes identically \cite[p. 17]{Lopez4}), it describes a plane or the 
\textit{Scherk surface} \cite{Scherk} 
\begin{equation*}
z\left( x,y\right) =\frac{1}{\lambda }\log \left\vert \frac{\cos \lambda x}{%
\cos \lambda y}\right\vert ,\text{ }\lambda \in \mathbb{R},\text{ }\lambda
\neq 0.
\end{equation*}%
If the translating curves lie in non-orthogonal planes, the translation
surface $M^{2}$ is locally given by $z=p\left( x\right) +q\left( y+\mu
x\right) ,$ $\mu \in \mathbb{R},\ \mu \neq 0,$ and so-called an \textit{%
affine translation surface} or a \textit{translation graph} \cite{Liu2}, 
\cite{Yang2}. A minimal affine translation surface is so-called \textit{%
affine Scherk surface} and given in the explicit form%
\begin{equation*}
z\left( x,y\right) =\frac{1}{\lambda }\log \left\vert \frac{\cos \lambda 
\sqrt{1+\mu ^{2}}x}{\cos \lambda \left( y+\mu x\right) }\right\vert .
\end{equation*}

L\'{o}pez \cite{Lopez7} obtained the singular minimal translation surfaces
in $\mathbb{R}^{3}$ of type $z=p\left( x\right) +q\left( y\right) $ with
respect to horizontal and vertical directions. This result was generalized
to higher dimensions in \cite{EAM}. For further study of singular minimal
surfaces, we refer to the L\'{o}pez's series of interesting papers on the
solutions of the Dirichlet problem for the $\alpha -$singular minimal
surface equation \cite{Lopez8}, the Lorentz-Minkowski counterpart of the
condition of singular minimality \cite{Lopez10}, the compact singular
minimal surfaces \cite{Lopez11} and the singular minimal surfaces with
density \cite{Lopez12}.

In this paper, we approach a singular minimal surface $M^{2}$ in $\mathbb{R}%
^{3}$ which is minimal\textit{.} We hereinafter assume that $\alpha \neq 0$
in \autoref{1}, otherwise any minimal surface obeys our approach, which is
trivial. Under this circumstance, \autoref{1} gives $\left\langle \mathbf{n%
}\left( \tilde{p}\right) ,\mathbf{v}\right\rangle =0,$ that is, the tangent
plane of $M^{2}$ at any point $\tilde{p}$ is parallel to $\mathbf{v}.$ In
such case, the surface $M^{2}$ belongs to the class of so-called \textit{%
constant angle surfaces} and has to be a plane parallel to $\mathbf{v}$ (see 
\cite[Proposition 9]{Munteanu}), yielding the following outcome.

\begin{proposition}\label{prop1}
Let $M^{2}$ be a singular minimal surface in $\mathbb{R}^{3}$ with respect
to an arbitrary vector $\mathbf{v}.$ If $M^{2}$ is minimal, then it is a
plane parallel to $\mathbf{v}.$
\end{proposition}

This result is changed when we modify \autoref{1} by using a special
semi-symmetric metric connection $\nabla $ (see \autoref{4}) on $\mathbb{R}%
^{3}$. In \autoref{sec3}, we prove that, besides planes, the singular minimal
surfaces which are minimal with respect to $\nabla $ are the generalized
cylinders, providing their explicit equations. It is also observed, in \autoref{sec3}, that this approach produces only trivial example when a special
semi-symmetric non-metric connection $D$ (see \autoref{23})\ is used.

We find the motivation in Wang's paper \cite{Wang} whose minimal translation
surfaces were obtained with respect to the connections $\nabla $ and $D$.\
The notion of a semi-symmetric metric (resp. non-metric) connection on a
Riemannian manifold were defined by Hayden \cite{Hayden} (resp. Agashe \cite%
{Agashe1}) and since then has been studied by many authors. Without giving a
complete list, we may refer to \cite{Agashe2}, \cite{Akyol}, \cite{Chaubey}, 
\cite{De}, \cite{Dogru},\cite{Duggal}, \cite{Gozutok}, \cite{Imai}, \cite{Murat}-\cite%
{Ozgur}, \cite{Yano2}-\cite{Yucesan2}. The present authors also obtained
singular minimal translation surfaces in $\mathbb{R}^{3}$ with respect to
the connections $\nabla $ and $D$ \cite{Erdur}.

Let $\mathbb{R}_{1}^{3}$ be the Lorentz-Minkowski $3-$space endowed with the
canonical Lorentzian metric $\left\langle \cdot ,\cdot \right\rangle
_{L}=dx^{2}+dy^{2}-dz^{2}$. Then we have \cite[Definition 1.1]{Lopez10}

\begin{definition}\label{def1}
Let $\mathbf{r}$ be a smooth immersion of a spacelike surface $M^{2}$ in the
halfspace $z>0$ of $\mathbb{R}_{1}^{3}$ and $\mathbf{n}$ unit normal vector
field on $M^{2}$ and $H$ the mean curvature. $M^{2}$ is called $\alpha -$%
singular maximal surface if satisfies 
\begin{equation}
H=-\alpha \frac{\left\langle \mathbf{n},\left( 0,0,1\right) \right\rangle
_{L}}{z},\text{ }\alpha \neq 0.  \label{2}
\end{equation}
\end{definition}

Due to the fact that the $z-$coordinate represents the time coordinate, the
concept of gravity has no meaning. Therefore, unlike the Riemannian case,
\autoref{2} describes only spacelike surfaces with prescribed angle
between $\mathbf{n}$ and the $z-$axis. Point out that $H$ is non-vanishing
in \autoref{2} if $\alpha \neq 0$ because $\left\langle \mathbf{n},\left(
0,0,1\right) \right\rangle _{L}\neq 0$ for timelike vectors $\mathbf{n}$ and 
$\left( 0,0,1\right) $ and so we can not adapt \autoref{2} to our study as
is. For this reason, we modify the concept as follows:

\begin{definition}\label{def2}
Let $\mathbf{r}$ be a smooth immersion of an oriented timelike surface $M^{2}
$ in $\mathbb{R}_{1}^{3}$ and $\mathbf{n}$ unit normal vector field on $M^{2}
$ and $H$ the mean curvature. Let $\mathbf{v}\in \mathbb{R}_{1}^{3},$ $%
\mathbf{v}\neq \mathbf{0},$ a spacelike vector non-parallel to $\mathbf{n}$
such that $\mathbf{n}$ and $\mathbf{v}$ span a spacelike 2-space. Then $M^{2}
$ is called singular minimal surface with respect to $\mathbf{v}$ if
satisfies%
\begin{equation}
2H=\alpha \frac{\left\langle \mathbf{n},\mathbf{v}\right\rangle _{L}}{%
\left\langle \mathbf{r},\mathbf{v}\right\rangle _{L}},\text{ }\alpha \in 
\mathbb{R},\text{ }\alpha \neq 0.  \label{3}
\end{equation}
\end{definition}

With \autoref{def2} we may view the singular minimal surface $M^{2}$ as a
timelike surface in $\mathbb{R}_{1}^{3}$ with prescribed Lorentz spacelike
angle between $\mathbf{n}$ and $\mathbf{v}$. If $M^{2}$ is minimal, it
follows from  \autoref{3} that $\left\langle \mathbf{n},\mathbf{v}%
\right\rangle _{L}=0,$ namely the angle is $\frac{\pi }{2},$ and, as in
Riemannian case, $M^{2}$ becomes a timelike constant angle surface which has
to be a plane (see \cite[Theorem 3.1]{Guler}), yielding the following
trivial outcome.

\begin{proposition}\label{prop2}
Let $M^{2}$ be a singular minimal surface in $\mathbb{R}_{1}^{3}$ with
respect to a spacelike vector $\mathbf{v}$. If $M^{2}$ is minimal, then it
is a plane parallel to $\mathbf{v}.$
\end{proposition}

In \autoref{sec4}, we also state non-trivial results in $\mathbb{R}_{1}^{3}$ for
singular minimal surfaces which are minimal with respect to the connections $%
\nabla $ and $D$ given by Eqs. \eqref{27} and \eqref{45}, respectively.

\section{Preliminaries}

Most of following notions can be found \cite{Chen}, \cite{O'Neill}, \cite%
{Yano1}.

Let $\left( \bar{M},\bar{g}\right) $ be a semi-Riemannian manifold and $\bar{%
\nabla}$ a linear connection on $\bar{M}.$ The \textit{torsion tensor field }%
$T$\textit{\ }of $\bar{\nabla}$ is defined by%
\begin{equation*}
T\left( \mathbf{\bar{x}},\mathbf{\bar{y}}\right) =\bar{\nabla}_{\mathbf{\bar{%
x}}}\mathbf{\bar{y}}-\bar{\nabla}_{\mathbf{\bar{x}}}\mathbf{\bar{y}}-\left[ 
\mathbf{\bar{x}},\mathbf{\bar{y}}\right] ,
\end{equation*}%
where $\mathbf{\bar{x}}$\textbf{\ }and\textbf{\ }$\mathbf{\bar{y}}$ are
vector fields on $\bar{M}$. A linear connection is called a \textit{%
semi-symmetric }(resp. \textit{non-})\textit{\ metric connection }if there
exist a $1-$form $\pi $ such that 
\begin{equation*}
T\left( \mathbf{\bar{x}},\mathbf{\bar{y}}\right) =\pi \left( \mathbf{\bar{y}}%
\right) \mathbf{\bar{x}}-\pi \left( \mathbf{\bar{x}}\right) \mathbf{\bar{y}},%
\text{ }\bar{\nabla}\bar{g}=0\text{ (resp. }\bar{\nabla}\bar{g}\neq 0\text{).%
}
\end{equation*}%
The linear connection $\bar{\nabla}$ is called \textit{Levi-Civita connection%
} if $T=0$ and $\bar{\nabla}\bar{g}=0$. We denote the Levi-Civita connection
by $\bar{\nabla}^{L}.$

Let $M$ be a semi-Riemannian submanifold of $\bar{M}$ and $\nabla ^{L}$ and $%
g$ the induced Levi-Civita connection and metric tensor, respectively. Then
the \textit{Gauss formula} follows%
\begin{equation*}
\bar{\nabla}_{\mathbf{x}}^{L}\mathbf{y=}\nabla _{\mathbf{x}}^{L}\mathbf{y}%
+h\left( \mathbf{x},\mathbf{y}\right) ,
\end{equation*}%
where $h$ is so-called \textit{second fundamental form} of $M$ and $\mathbf{x%
}$\textbf{\ }and\textbf{\ }$\mathbf{y}$ tangent vector fields to $M.$ Let $%
\left\{ \mathbf{f}_{1},...,\mathbf{f}_{n}\right\} $ be an orthonormal frame
on $M$ at any point $p\in M.$ Then the \textit{mean curvature vector} $%
\mathbf{H}\left( p\right) $ at $p$ is defined by 
\begin{equation*}
\mathbf{H}\left( p\right) =\frac{1}{n}\sum_{i=1}^{n}\varepsilon _{i}h\left( 
\mathbf{f}_{i},\mathbf{f}_{i}\right) ,
\end{equation*}%
where $\varepsilon _{i}=g\left( \mathbf{f}_{i},\mathbf{f}_{i}\right) $ and $%
n=\dim M.$ The length of mean curvature vector is called \textit{mean
curvature}. A semi-Riemannian submanifold is called \textit{minimal} if its
mean curvature vanishes identically.

Let $\bar{M}=\mathbb{R}_{1}^{3}$ be the Lorentz-Minkowski $3-$space and $%
\bar{g}=\left\langle \cdot ,\cdot \right\rangle _{L}=dx^{2}+dy^{2}-dz^{2}$.
A vector field $\mathbf{x}$ on $\mathbb{R}_{1}^{3}$ is said to be \textit{%
spacelike }(resp. \textit{timelike}) if $\mathbf{x}=0$ or $\left\langle 
\mathbf{x},\mathbf{x}\right\rangle _{L}>0$ (resp. $\left\langle \mathbf{x},%
\mathbf{x}\right\rangle _{L}<0$). A vector field $\mathbf{x}$ is said to be 
\textit{null} if $\left\langle \mathbf{x},\mathbf{x}\right\rangle _{L}=0$
and $\mathbf{x}\neq 0.$ A timelike vector $\mathbf{x}=\left( a,b,c\right) $
is said to be \textit{future pointing} (resp. \textit{past pointing}) if $%
c>0 $ (resp. $c<0$). A \textit{Lorentz timelike angle} $\theta $ between two
future (past) pointing timelike vectors $\mathbf{x}$ and $\mathbf{y}$ is
associated with \cite{Fu}%
\begin{equation*}
\left\vert \left\langle \mathbf{x},\mathbf{y}\right\rangle _{L}\right\vert =%
\sqrt{\left\vert \left\langle \mathbf{x},\mathbf{x}\right\rangle
_{L}\right\vert }\sqrt{\left\vert \left\langle \mathbf{y},\mathbf{y}%
\right\rangle _{L}\right\vert }\cosh \theta .
\end{equation*}%
A \textit{Lorentz spacelike angle} $\theta $ between two spacelike vectors $%
\mathbf{x}$ and $\mathbf{y}$ spanning a spacelike vector subspace ($\mathbb{R%
}_{1}^{3}$ induces a Riemannian metric on it) is associated with \cite{Fu}%
\begin{equation*}
\left\vert \left\langle \mathbf{x},\mathbf{y}\right\rangle _{L}\right\vert =%
\sqrt{\left\vert \left\langle \mathbf{x},\mathbf{x}\right\rangle
_{L}\right\vert }\sqrt{\left\vert \left\langle \mathbf{y},\mathbf{y}%
\right\rangle _{L}\right\vert }\cos \theta .
\end{equation*}

Let $M^{2}$ be an immersed surface into $\mathbb{R}_{1}^{3}$. The surface $%
M^{2}$ is said to be \textit{spacelike} (resp. \textit{timelike}) if all
tangent planes of $M^{2}$ are spacelike (resp. timelike). For such a
spacelike (resp. timelike) surface, we have the decomposition $\mathbb{R}%
_{1}^{3}=T_{p}M^{2}\oplus \left( T_{p}M^{2}\right) ^{\perp },$\ where $%
T_{p}M^{2}$ is the tangent plane of $M^{2}$ at the point $p.$ Notice that $%
\left( T_{p}M^{2}\right) ^{\perp }$ is a timelike (resp. spacelike) $1-$%
space of $\mathbb{R}_{1}^{3}$. A \textit{Gauss map }$\mathbf{n}$ of $M^{2}$
is a smooth map $\mathbf{n}:M^{2}\rightarrow \mathbb{R}_{1}^{3},$ $%
\left\vert \left\langle \mathbf{n},\mathbf{n}\right\rangle _{L}\right\vert
=1.$

We finish this section remarking that a spacelike (resp. timelike) surface
in $\mathbb{R}_{1}^{3}$ is locally a graph of a smooth function $u\left(
x,y\right) $ (resp. $u\left( x,z\right) $ or $u\left( y,z\right) $) \cite[%
Proposition 3.3]{Lopez2}.

\section{Singular minimal surfaces in $\mathbb{R}^{3}$}\label{sec3}

\subsection{$\protect\nabla -$Singular minimal surfaces}

Let $\nabla ^{L}$ be the Levi-Civita connection on $\mathbb{R}^{3}$ and $%
\left\{ \mathbf{e}_{1},\mathbf{e}_{2},\mathbf{e}_{3}\right\} $ the standard
basis on $\mathbb{R}^{3}$ and $\mathbf{x},\mathbf{y}$ tangent vector fields
to $\mathbb{R}^{3}$. Consider the following semi-symmetric metric connection
on $\mathbb{R}^{3}$ \cite{Wang}%
\begin{equation}
\nabla _{\mathbf{x}}\mathbf{y}=\nabla _{\mathbf{x}}^{L}\mathbf{y}%
+\left\langle \mathbf{y},\mathbf{e}_{3}\right\rangle \mathbf{x}-\left\langle 
\mathbf{x},\mathbf{y}\right\rangle \mathbf{e}_{3}.  \label{4}
\end{equation}%
The nonzero derivatives are 
\begin{equation*}
\nabla _{\mathbf{e}_{1}}\mathbf{e}_{1}=-\mathbf{e}_{3},\text{ }\nabla _{%
\mathbf{e}_{1}}\mathbf{e}_{3}=\mathbf{e}_{1},\text{ }\nabla _{\mathbf{e}_{2}}%
\mathbf{e}_{2}=-\mathbf{e}_{3},\text{ }\nabla _{\mathbf{e}_{2}}\mathbf{e}%
_{3}=\mathbf{e}_{2}.
\end{equation*}

\begin{definition}\label{def3}
Let $\mathbf{r}$ be a smooth immersion of an oriented surface $M^{2}$ into $%
\mathbb{R}^{3}$ and $\mathbf{n}$ unit normal vector field on $M^{2}$ and $%
H^{\nabla }$ the mean curvature with respect to $\nabla .$ Let $\mathbf{v}%
\in \mathbb{R}^{3},$ $\mathbf{v}\neq \mathbf{0},$ a unit fixed vector
non-parallel to $\mathbf{n}.$ The surface $M^{2}$ is called $\nabla -$s%
\textit{ingular minimal surface} with respect to $\mathbf{v}$ if holds 
\begin{equation}
2H^{\nabla }=\alpha \frac{\left\langle \mathbf{n},\mathbf{v}\right\rangle }{%
\left\langle \mathbf{r},\mathbf{v}\right\rangle },\text{ }\alpha \in \mathbb{%
R},\text{ }\alpha \neq 0.  \label{5}
\end{equation}
\end{definition}

In particular, the surface $M^{2}$ is said to be $\nabla -$\textit{minimal}
if $H^{\nabla }=0$. With \autoref{def3}, we first observe the $\nabla -$%
singular minimal surfaces of type $z=u\left( x,y\right) $ which are $\nabla
- $minimal.

\begin{theorem}\label{teo1}
Let $M^{2}$ be a $\nabla -$singular minimal surface in $\mathbb{R}^{3}$ of
type $z=u\left( x,y\right) $ with respect to a unit vector $\mathbf{v}%
=\left( a,b,c\right) ,$ $a^{2}+b^{2}\neq 0.$ If $M^{2}$ is $\nabla -$%
minimal, then one of the following happens

\begin{enumerate}
\item $\mathbf{v}=\left( 0,b\neq 0,c\right) $ and 
\begin{equation*}
u\left( x,y\right) =\frac{c}{b}y+\frac{1}{2b^{2}}\ln \left[ \cos \left(
2bx+\lambda _{1}\right) \right] +\lambda _{2};
\end{equation*}

\item $\mathbf{v}=\left( a\neq 0,0,c\right) $ and 
\begin{equation*}
u\left( x,y\right) =\frac{c}{a}x+\frac{1}{2a^{2}}\ln \left[ \cos \left(
2ay+\lambda _{3}\right) \right] +\lambda _{4};
\end{equation*}

\item $\mathbf{v}=\left( a,b,c\right) ,$ $ab\neq 0,$ and 
\begin{equation*}
u\left( x,y\right) =\frac{c}{a}x-\frac{1}{2\left( a^{2}+b^{2}\right) }\ln %
\left[ \cos \left( -2\left\vert a\right\vert \left( y-\frac{b}{a}x\right)
+\lambda _{5}\right) \right] +\frac{bc}{a^{2}+b^{2}}\left( y-\frac{b}{a}%
x\right) +\lambda _{6},
\end{equation*}%
where $\lambda _{1},...,\lambda _{6}\in \mathbb{R}$.
\end{enumerate}
\end{theorem}

\begin{proof}
The unit normal vector field on $M^{2}$ is%
\begin{equation*}
\mathbf{n}=\frac{-u_{x}\mathbf{e}_{1}-u_{y}\mathbf{e}_{2}+\mathbf{e}_{3}}{%
\sqrt{1+\left( u_{x}\right) ^{2}+\left( u_{y}\right) ^{2}}},
\end{equation*}%
where $u_{x}=\frac{\partial u}{\partial x}$ and so. Suppose that $M^{2}$ is $%
\nabla -$minimal. Due to $\alpha \neq 0,$ \autoref{5} gives $\left\langle 
\mathbf{n},\mathbf{v}\right\rangle =0$ and%
\begin{equation}
au_{x}+bu_{y}=c.  \label{6}
\end{equation}%
The condition of $\nabla -$minimality yields%
\begin{equation}
\left[ 1+\left( u_{y}\right) ^{2}\right] u_{xx}-2u_{x}u_{y}u_{xy}+\left[
1+\left( u_{x}\right) ^{2}\right] u_{yy}-2\left[ 1+\left( u_{x}\right)
^{2}+\left( u_{y}\right) ^{2}\right] =0.  \label{7}
\end{equation}%
We distinguish several cases: the first case is that $a=0$. Then \autoref{6} follows $u\left( x,y\right) =f\left( x\right) +\frac{c}{b}y,$ for
an arbitrary smooth function $f.$ Considering this into \autoref{7} leads
to%
\begin{equation}
\frac{bf^{\prime \prime }}{1+\left( bf^{\prime }\right) ^{2}}=2b,  \label{8}
\end{equation}%
where $f^{\prime }=\frac{df}{dx}$ and so. The first statement of the theorem
is obtained by integrating \autoref{8}. The roles of $x$ and $y$ in \autoref{7} are symmetric and hence we may conclude the second statement of
the theorem by similar steps when $a\neq 0$ and $b=0.$ The last case is that 
$ab\neq 0.$ Then the solution to \autoref{6} is given by 
\begin{equation}
u\left( x,y\right) =\frac{c}{a}x+g\left( y-\frac{b}{a}x\right) ,  \label{9}
\end{equation}%
for a smooth function $g.$ Substituting \autoref{9} into \autoref{7}
follows 
\begin{equation}
g^{\prime \prime }-2\left[ a^{2}+\left( c-bg^{\prime }\right) ^{2}+\left(
ag^{\prime }\right) ^{2}\right] =0,  \label{10}
\end{equation}%
for $g^{\prime }=\frac{dg}{d\tilde{y}},$ $g^{\prime \prime }=\frac{d^{2}g}{d%
\tilde{y}^{2}},$ $\tilde{y}=y-\frac{b}{a}x.$ \autoref{10} can be rewritten
as
\begin{equation}
\frac{\left( a^{2}+b^{2}\right) g^{\prime \prime }}{a^{2}+\left( bc-\left(
a^{2}+b^{2}\right) g^{\prime }\right) ^{2}}=2.  \label{11}
\end{equation}%
The proof is completed by integrating \autoref{11}.
\end{proof}

\begin{remark}\label{rem1}
The surface given in the first statement of \autoref{teo1} is a generalized
cylinder (see \cite[p. 439]{Gray}) and may be written parametrically 
\begin{equation*}
\mathbf{r}\left( x,y\right) =\left( x,0,\frac{1}{2b^{2}}\ln \left[ \cos
\left( 2bx+\lambda _{1}\right) \right] +\lambda _{2}\right) +y\left( 0,1,%
\frac{c}{b}\right) .
\end{equation*}%
This is a $\nabla -$minimal translation surface of type $z=p\left( x\right)
+q\left( y\right) $ which was already found by Wang \cite{Wang}. The same
may be concluded for the above second statement. However, the surface
described in the last statement of \autoref{teo1} is the generalized cylinder
parametrically written by 
\begin{equation*}
\mathbf{r}\left( x,\tilde{y}\right) =x\left( 1,\frac{b}{a},\frac{c}{a}%
\right) +\left( 0,\tilde{y},g\left( \tilde{y}\right) \right) ,
\end{equation*}%
where $\tilde{y}=y-\frac{b}{a}x.$ Due to $b\neq 0,$ it belongs to the class
of affine translation surfaces and a new example of $\nabla -$minimal
surfaces.
\end{remark}

In the following we classify $\nabla -$singular minimal surfaces in $\mathbb{%
R}^{3}$ of type $y=u\left( x,z\right) $ which are $\nabla -$minimal.

\begin{theorem}\label{teo2}
Let $M^{2}$ be a $\nabla -$singular minimal surface in $\mathbb{R}^{3}$ of
type $y=u\left( x,z\right) $ with respect to a unit vector $\mathbf{v}%
=\left( a,b,c\right) ,$ $a^{2}+c^{2}\neq 0.$ If $M^{2}$ is $\nabla -$%
minimal, then one of the following happens

\begin{enumerate}
\item $M^{2}$ is a plane parallel to the vector $\left( 0,0,1\right) $;

\item $\mathbf{v}=\left( 0,b,c\right) ,$ $bc\neq 0,$ and 
\begin{equation*}
u\left( x,z\right) =\frac{b}{c}z+\frac{1}{2bc}\ln \left[ \cos \left(
2bx+\lambda _{1}\right) \right] +\lambda _{2};
\end{equation*}

\item $\mathbf{v}=\left( a,b,0\right) ,$ $a\neq 0,$ and 
\begin{equation*}
u\left( x,z\right) =\frac{b}{a}x\pm \frac{1}{2\left\vert a\right\vert }%
\arctan \left( \frac{1}{\left\vert a\lambda _{2}\right\vert }\sqrt{%
e^{4z}-a^{2}}\right) +\lambda _{3};
\end{equation*}

\item $\mathbf{v}=\left( a,0,c\right) ,$ $ac\neq 0,$ and%
\begin{equation*}
u\left( x,z\right) =\pm \frac{1}{2\left\vert a\right\vert }\arctan \left( 
\frac{1}{\left\vert \lambda _{4}\right\vert }\sqrt{e^{4a^{2}\left( z-\frac{c%
}{a}x\right) }-\lambda _{4}^{2}}\right) +\lambda _{5},\text{ }\lambda
_{4}\neq 0;
\end{equation*}

\item $\mathbf{v}=\left( a,b,c\right) ,$ $ac\neq 0,$ and%
\begin{equation*}
u\left( x,z\right) =\frac{b}{a}x+h\left( z-\frac{c}{a}x\right) ,
\end{equation*}%
where $h$ is a smooth function satisfying%
\begin{equation*}
\left. 
\begin{array}{c}
z-\frac{c}{a}x=\frac{1}{2\left\vert a\right\vert \left( a^{2}+c^{2}\right)
\left( a^{2}+b^{2}c^{2}\right) }\left\{ bc\left( 2\left\vert a\right\vert
h+\lambda _{6}\right) -\right.  \\ 
\left. -\left\vert a\right\vert \ln \left[ bc\cos \left( 2\left\vert
a\right\vert h+\lambda _{6}\right) -\left\vert a\right\vert \sin \left(
2\left\vert a\right\vert h+\lambda _{6}\right) \right] \right\} +\lambda
_{7},%
\end{array}%
\right. 
\end{equation*}%
for $\lambda _{1},...,\lambda _{7}\in \mathbb{R}.$
\end{enumerate}
\end{theorem}

\begin{proof}
Let $M^{2}$ be locally given by 
\begin{equation*}
(x,z)\longmapsto \mathbf{r}\left( x,z\right) =\left( x,u\left( x,z\right)
,z\right) ,
\end{equation*}%
for a smooth function $u=u\left( x,z\right) $. The normal vector field on $%
M^{2}$ is%
\begin{equation}
\mathbf{n}=\frac{u_{x}\mathbf{e}_{1}-\mathbf{e}_{2}+u_{z}\mathbf{e}_{3}}{%
\sqrt{1+\left( u_{x}\right) ^{2}+\left( u_{z}\right) ^{2}}}.  \label{12}
\end{equation}%
Because $M^{2}$ is $\nabla -$singular minimal, we get \autoref{5}. Assume
that $M^{2}$ is $\nabla -$minimal. Due to $\alpha \neq 0,$ Eqs. \eqref{5}
and \eqref{12} follow $\left\langle \mathbf{n},\mathbf{v}\right\rangle =0$
and 
\begin{equation}
au_{x}+cu_{z}=b.  \label{13}
\end{equation}%
Remark also that we may write $\mathbf{v}=a\mathbf{r}_{x}+c\mathbf{r}_{z},$
which means that the tangent plane of $M$ at any point is parallel to $%
\mathbf{v}.$ The condition of $\nabla -$minimality leads to%
\begin{equation}
\left[ 1+\left( u_{z}\right) ^{2}\right] u_{xx}-2u_{x}u_{z}u_{xz}+\left[
1+\left( u_{x}\right) ^{2}\right] u_{zz}+2\left[ 1+\left( u_{x}\right)
^{2}+\left( u_{z}\right) ^{2}\right] u_{z}=0.  \label{14}
\end{equation}%
We distinguish several cases:

\begin{enumerate}
\item $a=0,$ $c\neq 0.$ Then \autoref{13} gives $u_{z}=\frac{b}{c}$ and so \autoref{14} turns $M^{2}$ to a plane parallel to $\mathbf{v}$ if $b=0.$
Otherwise, $b\neq 0,$ the solution to \autoref{13} is given by $u\left(
x,z\right) =\frac{b}{c}z+f\left( x\right) ,$ for an arbitrary smooth
function $f.$ Hence \autoref{14} reduces to%
\begin{equation}
\frac{cf^{\prime \prime }}{1+\left( cf^{\prime }\right) ^{2}}=-2b,
\label{15}
\end{equation}%
where $f^{\prime }=\frac{df}{dx},$ etc. The second statement of the theorem
is obtained by integrating \autoref{15}.

\item $a\neq 0,$ $c=0.$ Then \autoref{13} gives $u\left( x,z\right) =\frac{%
b}{a}x+g\left( z\right) $ for an arbitrary smooth function $g$ and so \autoref{14} may be written as
\begin{equation}
\frac{g^{\prime \prime }}{g^{\prime }}-\frac{a^{2}g^{\prime }g^{\prime
\prime }}{1+\left( ag^{\prime }\right) ^{2}}=-2,  \label{16}
\end{equation}%
where $g^{\prime }=\frac{dg}{dz},$ etc. Integrating \autoref{16}, we
obtain the third statement of the theorem.

\item $ac\neq 0.$ The solution to \autoref{13} is%
\begin{equation}
u\left( x,z\right) =\frac{b}{a}x+h\left( z-\frac{c}{a}x\right) ,  \label{17}
\end{equation}%
for an arbitrary smooth function $h$. By plugging the partial derivatives of
\autoref{17}\ into  \autoref{14}, we write%
\begin{equation}
h^{\prime \prime }+2\left[ a^{2}+\left( b-ch^{\prime }\right) ^{2}+\left(
ah^{\prime }\right) ^{2}\right] h^{\prime }=0,  \label{18}
\end{equation}%
where $h^{\prime }=\frac{dh}{d\tilde{z}},$ $h^{\prime \prime }=\frac{d^{2}h}{%
d\tilde{z}^{2}},$ $\tilde{z}=z-\frac{c}{a}x.$ We have two subcases: the
first subcase is that $b=0.$ Then Eq. \eqref{18} may be rewritten as%
\begin{equation}
\frac{h^{\prime \prime }}{h^{\prime }}-\frac{h^{\prime }h^{\prime \prime }}{%
a^{2}+\left( h^{\prime }\right) ^{2}}=-2a^{2}.  \label{19}
\end{equation}%
The fourth statement of the theorem is proved by integrating \autoref{19}.
The second subcase is $b\neq 0.$ Hence, we may write \autoref{18} as
\begin{equation}
\frac{-\left( a^{2}+c^{2}\right) h^{\prime \prime }}{a^{2}+\left( bc-\left(
a^{2}+c^{2}\right) h^{\prime }\right) ^{2}}=2h^{\prime }.  \label{21}
\end{equation}%
A first integration of \autoref{21} yields
\begin{equation}
\frac{\left( a^{2}+c^{2}\right) dh}{-\left\vert a\right\vert \tan \left(
2\left\vert a\right\vert h+\lambda \right) +bc}=d\tilde{z},  \label{22}
\end{equation}%
for $\lambda \in \mathbb{R}.$ By a first integration of \autoref{22}, we finish the proof.
\end{enumerate}
\end{proof}

\begin{remark}\label{rem2}
The surfaces given in the second and third statements of \autoref{teo2} are $%
\nabla -$minimal generalized cylinders and are examples of $\nabla -$minimal
translation surfaces of type $y=p\left( x\right) +q\left( z\right) ,$ which
was found by Wang \cite{Wang}. However, the surfaces given in the last two
statements of \autoref{teo2} are a $\nabla -$minimal affine translation surface.
\end{remark}

Lastly, we deal with a surface $M^{2}$ of type $x=u\left( y,z\right) $. The
unit normal vector field on $M^{2}$ is%
\begin{equation*}
\mathbf{n}=\frac{\mathbf{e}_{1}-u_{y}\mathbf{e}_{2}-u_{z}\mathbf{e}_{3}}{%
\sqrt{1+\left( u_{y}\right) ^{2}+\left( u_{z}\right) ^{2}}}.
\end{equation*}%
Suppose that $M^{2}$ is $\nabla -$singular minimal with respect to the
vector $\mathbf{v}=\left( a,b,c\right) $. The mean curvature is same as that
of the surface of type $y=u\left( x,z\right) .$ If $M^{2}$ is also $\nabla -$%
minimal, then \autoref{3} gives%
\begin{equation*}
bu_{y}+cu_{z}=a,
\end{equation*}%
where $b^{2}+c^{2}\neq 0.$ Therefore, without giving a proof, we may state a
similar result for those surfaces of type $x=u\left( y,z\right) $ to \autoref{teo2} by replacing $x$ with $y$ and $a$ with $b.$

\subsection{$D-$Singular minimal surfaces}

Let $D$ be the semi-symmetric non-metric connection on $\mathbb{R}^{3}$
given by \cite{Wang}%
\begin{equation}
D_{\mathbf{x}}\mathbf{y}=\nabla _{\mathbf{x}}^{L}\mathbf{y}+\left\langle 
\mathbf{y},\mathbf{e}_{3}\right\rangle \mathbf{x},  \label{23}
\end{equation}%
where $\mathbf{x},\mathbf{y}$ are tangent vector fields to $\mathbb{R}^{3}.$
The nonzero derivatives are 
\begin{equation*}
D_{\mathbf{e}_{1}}\mathbf{e}_{3}=\mathbf{e}_{1},\text{ }D_{\mathbf{e}_{2}}%
\mathbf{e}_{3}=\mathbf{e}_{2},\text{ }D_{\mathbf{e}_{3}}\mathbf{e}_{3}=%
\mathbf{e}_{3}.
\end{equation*}

\begin{definition}\label{def4}
Let $\mathbf{r}$ be a smooth immersion of an oriented surface $M^{2}$ into $%
\mathbb{R}^{3}$ and $\mathbf{n}$ unit normal vector field on $M^{2}$ and $%
H^{D}$ denote the mean curvature with respect to $D.$ Let $\mathbf{v}\in 
\mathbb{R}^{3},$ $\mathbf{v}\neq \mathbf{0},$ a unit fixed vector
non-parallel to $\mathbf{n}.$ The surface $M^{2}$ is called $D-$s\textit{%
ingular minimal surface} with respect to $\mathbf{v}$ if holds 
\begin{equation}
2H^{D}=\alpha \frac{\left\langle \mathbf{n},\mathbf{v}\right\rangle }{%
\left\langle \mathbf{r},\mathbf{v}\right\rangle },\text{ }\alpha \in \mathbb{%
R},\text{ }\alpha \neq 0.  \label{24}
\end{equation}
\end{definition}

In particular, the surface $M^{2}$ is said to be $D-$\textit{minimal} if $%
H^{D}=0$. We first consider the $D-$singular minimal surfaces of type $%
z=u\left( x,y\right) $ which are $D-$minimal. Hence \autoref{24} gives $
\left\langle \mathbf{n},\mathbf{v}\right\rangle =0$ and 
\begin{equation}
au_{x}+bu_{y}=c,  \label{25}
\end{equation}%
where $\mathbf{v}=\left( a,b,c\right) $ and $a^{2}+b^{2}\neq 0.$ Morever the
condition of $D-$minimality yields%
\begin{equation}
\left[ 1+\left( u_{y}\right) ^{2}\right] u_{xx}-2u_{x}u_{y}u_{xy}+\left[
1+\left( u_{x}\right) ^{2}\right] u_{yy}=0,  \label{26}
\end{equation}%
where the roles of $x$ and $y$ are symmetric. If $a=0$, then \autoref{25}
follows $u\left( x,y\right) =f\left( x\right) +\frac{c}{b}y,$ for an
arbitrary smooth function $f.$ Considering this into \autoref{26} yields $
\frac{1}{b^{2}}\frac{d^{2}f}{dx^{2}}=0,$ which leads $M^{2}$ to be a plane
parallel to $\mathbf{v}.$ By symmetry, we may obtain same obtain when $a\neq
0$ and $b=0$. Let $ab\neq 0.$ Then the solution to \autoref{25} is $
u\left( x,y\right) =\frac{c}{a}x+g\left( y-\frac{b}{a}x\right) ,$ for an
arbitrary smooth function $f.$ After substituting its partial derivatives
into \autoref{26}, we conclude $\frac{1}{a^{2}}\frac{d^{2}g}{d\tilde{y}^{2}%
}=0,$ $\tilde{y}=y-\frac{b}{a}x,$ yielding that $M$ is a plane parallel to $%
\mathbf{v}.$

Therefore we state the following

\begin{theorem}\label{teo3}
Let $M^{2}$ be a $D-$singular minimal surface in $\mathbb{R}^{3}$ of type $%
z=u\left( x,y\right) $ with respect to a unit vector $\mathbf{v}=\left(
a,b,c\right) ,$ $a^{2}+b^{2}\neq 0.$ If $M^{2}$ is $D-$minimal, then it is a
plane parallel to $\mathbf{v}.$
\end{theorem}

When we take surfaces of type $y=u\left( x,z\right) $ and $x=u\left(
y,z\right) ,$ we get similar equations to Eqs. \eqref{25} and \eqref{26} and
thus the above result remains true for those surfaces as well.

\section{Singular minimal surfaces in $\mathbb{R}_{1}^{3}$}\label{sec4}

\subsection{$\protect\nabla -$Singular minimal surfaces}

Let $\nabla ^{L}$ be the Levi-Civita connection $\mathbb{R}_{1}^{3}$ and $%
\mathbf{x},\mathbf{y}$ tangent vector fields to $\mathbb{R}_{1}^{3}$.
Consider the following semi-symmetric metric connection on $\mathbb{R}%
_{1}^{3}$ \cite{Wang}%
\begin{equation}
\nabla _{\mathbf{x}}\mathbf{y}=\nabla _{\mathbf{x}}^{L}\mathbf{y}%
+\left\langle \mathbf{y},\mathbf{e}_{3}\right\rangle _{L}\mathbf{x}%
-\left\langle \mathbf{x},\mathbf{y}\right\rangle _{L}\mathbf{e}_{3}.
\label{27}
\end{equation}%
The nonzero derivatives are 
\begin{equation*}
\nabla _{\mathbf{e}_{1}}\mathbf{e}_{1}=-\mathbf{e}_{3},\text{ }\nabla _{%
\mathbf{e}_{1}}\mathbf{e}_{3}=-\mathbf{e}_{1},\text{ }\nabla _{\mathbf{e}%
_{2}}\mathbf{e}_{2}=-\mathbf{e}_{3},\text{ }\nabla _{\mathbf{e}_{2}}\mathbf{e%
}_{3}=-\mathbf{e}_{2}.
\end{equation*}

\begin{definition}\label{def5}
Let $\mathbf{r}$ be a smooth immersion of an oriented timelike surface $M^{2}
$ in $\mathbb{R}_{1}^{3}$ and $\mathbf{n}$ unit normal vector field on $M^{2}
$ and $H^{\nabla }$ the mean curvature of $M^{2}$ with respect to $\nabla .$
Let $\mathbf{v\neq 0}\in \mathbb{R}_{1}^{3}$ a unit fixed spacelike vector
non-parallel to $\mathbf{n}$ such that $\mathbf{n}$ and $\mathbf{v}$ span a
spacelike 2-space. $M^{2}$ is called $\nabla -$\textit{singular minimal
surface} with respect to $\mathbf{v}$ if satisfies%
\begin{equation}
2H^{\nabla }=\alpha \frac{\left\langle \mathbf{n},\mathbf{v}\right\rangle
_{L}}{\left\langle \mathbf{r},\mathbf{v}\right\rangle _{L}},\text{ }\alpha
\in \mathbb{R},\text{ }\alpha \neq 0.  \label{28}
\end{equation}
\end{definition}

The surface $M^{2}$ is called $\nabla -$\textit{minimal} if $H^{\nabla }=0.$
With \autoref{def5}, we classify the $\nabla -$singular minimal surfaces of
type $z=u\left( x,y\right) ,$ which are $\nabla -$minimal.

\begin{theorem}\label{teo4}
Let $M^{2}$ be a $\nabla -$singular minimal surface in $\mathbb{R}_{1}^{3}$
of type $z=u\left( x,y\right) $ with respect to a unit spacelike vector $%
\mathbf{v}=\left( a,b,c\right) $. If $M^{2}$ is $\nabla -$minimal, then one
of the following happens

\begin{enumerate}
\item $\mathbf{v}=\left( 0,b\neq 0,c\right) $ and 
\begin{equation*}
u\left( x,y\right) =\frac{c}{b}y+\frac{1}{2b^{2}}\ln \left[ \cosh \left(
2bx+\lambda _{1}\right) \right] +\lambda _{2};
\end{equation*}

\item $\mathbf{v}=\left( a\neq 0,0,c\right) $ and 
\begin{equation*}
u\left( x,y\right) =\frac{c}{a}x+\frac{1}{2a^{2}}\ln \left[ \cosh \left(
2ay+\lambda _{3}\right) \right] +\lambda _{4};
\end{equation*}

\item $\mathbf{v}=\left( a,b,c\right) ,$ $ab\neq 0,$ and 
\begin{equation*}
u\left( x,y\right) =\frac{c}{a}x+\frac{bc}{a^{2}+b^{2}}\left( y-\frac{b}{a}%
x\right) +\frac{1}{2\left( a^{2}+b^{2}\right) }\ln \left[ \cosh \left(
-2\left\vert a\right\vert \left\{ y-\frac{b}{a}x\right\} +\lambda
_{5}\right) \right] +\lambda _{6},
\end{equation*}%
where $\lambda _{1},...,\lambda _{6}\in \mathbb{R}$, $\lambda _{5}\neq 0.$
\end{enumerate}
\end{theorem}

\begin{proof}
The unit spacelike normal vector field on $M^{2}$ is%
\begin{equation*}
\mathbf{n}=\frac{-u_{x}\mathbf{e}_{1}-u_{y}\mathbf{e}_{2}-\mathbf{e}_{3}}{%
\sqrt{-1+\left( u_{x}\right) ^{2}+\left( u_{y}\right) ^{2}}},
\end{equation*}
Suppose that $M^{2}$ is $\nabla -$minimal. Due to $\alpha \neq 0,$ \autoref{28} gives $\left\langle \mathbf{n},\mathbf{v}\right\rangle =0$ and%
\begin{equation}
au_{x}+bu_{y}=c.  \label{29}
\end{equation}%
The condition of $\nabla -$minimality yields%
\begin{equation}
\left[ 1-\left( u_{y}\right) ^{2}\right] u_{xx}+2u_{x}u_{y}u_{xy}+\left[
1-\left( u_{x}\right) ^{2}\right] u_{yy}+2\left[ -1+\left( u_{x}\right)
^{2}+\left( u_{y}\right) ^{2}\right] =0.  \label{30}
\end{equation}%
We distinguish several cases: the first case is that $a=0$. Then \autoref{29} follows $u\left( x,y\right) =f\left( x\right) +\frac{c}{b}y,$ for
an arbitrary smooth function $f.$ Considering this into \autoref{30} yields
\begin{equation}
\frac{bf^{\prime \prime }}{1-\left( bf^{\prime }\right) ^{2}}=2b,  \label{31}
\end{equation}
where $f^{\prime }=\frac{df}{dx}$ and so. The first statement of the theorem
is derived by integrating \autoref{31}. The roles of $x$ and $y$ in \autoref{30} are symmetric and hence we may conclude the second statement of
the theorem by similar steps when $a\neq 0$ and $b=0.$ The last case is that 
$ab\neq 0.$ Then the solution to \autoref{29} is given by 
\begin{equation}
u\left( x,y\right) =\frac{c}{a}x+g\left( y-\frac{b}{a}x\right) ,  \label{32}
\end{equation}%
for a smooth function $g.$ Substituting \autoref{32} into \autoref{30} follows 
\begin{equation}
g^{\prime \prime }+2\left[ -a^{2}+\left( c-bg^{\prime }\right) ^{2}+\left(
ag^{\prime }\right) ^{2}\right] =0,  \label{33}
\end{equation}%
where $g^{\prime }=\frac{dg}{d\tilde{y}},$ $g^{\prime \prime }=\frac{d^{2}g}{%
d\tilde{y}^{2}},$ $\tilde{y}=y-\frac{b}{a}x.$ \autoref{33} may be rewritten as
\begin{equation}
\frac{-\left( a^{2}+b^{2}\right) g^{\prime \prime }}{a^{2}-\left[ bc-\left(
a^{2}+b^{2}\right) g^{\prime }\right] ^{2}}=-2.  \label{34}
\end{equation}%
The proof is completed by integrating \autoref{34}.
\end{proof}

\begin{remark}\label{rem3}
The last statement of \autoref{teo4} is a new example in $\mathbb{R}_{1}^{3}$ of $%
\nabla -$minimal surfaces while the first two statements are $\nabla -$%
minimal translation surfaces, introduced by Wang \cite{Wang}.
\end{remark}

In the following we classify $\nabla -$singular minimal surfaces in $\mathbb{%
R}_{1}^{3}$ of type $y=u\left( x,z\right) $ which are $\nabla -$minimal.

\begin{theorem}\label{teo5}
Let $M^{2}$ be a $\nabla -$singular minimal surface in $\mathbb{R}_{1}^{3}$
of type $y=u\left( x,z\right) $ with respect to a unit spacelike vector $%
\mathbf{v}=\left( a,b,c\right) ,$ $a^{2}+c^{2}\neq 0$. If $M^{2}$ is $\nabla
-$minimal, then one of the following happens

\begin{enumerate}
\item $M^{2}$ is a plane parallel to $\mathbf{v}=\left( a,b,0\right) ,$ $%
a\neq 0;$

\item $\mathbf{v}=\left( 0,b,c\right) ,$ $bc\neq 0$ and 
\begin{equation*}
u\left( x,y\right) =\frac{b}{c}z+\frac{1}{2bc}\ln \left[ \cosh \left(
2bx+\lambda _{1}\right) \right] +\lambda _{2};
\end{equation*}

\item $\mathbf{v}=\left( a,b,0\right) ,$ $a\neq 0,$ and 
\begin{equation*}
u\left( x,z\right) =\frac{b}{a}x\pm \frac{1}{2\left\vert a\right\vert }\sinh
^{-1}\left( \lambda _{3}e^{2z}\right) +\lambda _{4},\text{ }\lambda _{3}\neq
0;
\end{equation*}

\item $\mathbf{v}=\left( a,0,c\right) ,$ $ac\neq 0,$ and%
\begin{equation*}
u\left( x,z\right) =\pm \frac{1}{2\left\vert a\right\vert }\sinh ^{-1}\left[
\left\vert \lambda _{5}\right\vert e^{2a^{2}\left( z-\frac{c}{a}x\right) }%
\right] +\lambda _{6},\text{ }\lambda _{5}\neq 0;
\end{equation*}

\item $\mathbf{v}=\left( a,\pm 1,c\right) ,$ $a=\pm c,$ $c\neq 0,$ and%
\begin{equation*}
u\left( x,z\right) =\frac{\pm 1}{c}x\pm \frac{1}{4c}\ln \left[ 1\pm 2\lambda
_{7}e^{2\left( 1+c^{2}\right) \left( z\pm x\right) }\right] +\lambda _{8},%
\text{ }\lambda _{7}\neq 0;
\end{equation*}

\item $\mathbf{v}=\left( a,b,c\right) ,$ $abc\neq 0,$ and%
\begin{equation*}
u\left( x,z\right) =\frac{b}{a}x+h\left( z-\frac{c}{a}x\right) ,
\end{equation*}%
where $h$ is a smooth function satisfying%
\begin{gather*}
z-\frac{c}{a}x=\frac{-bc\left( c^{2}-a^{2}\right) }{2\left\vert a\right\vert
\left( a^{2}-b^{2}c^{2}\right) }\left( 2\left\vert a\right\vert h+\lambda
_{9}\right) - \\
-\frac{c^{2}-a^{2}}{2\left( a^{2}-b^{2}c^{2}\right) }\ln \left[ bc\cosh
\left( 2\left\vert a\right\vert h+\lambda _{9}\right) -\left\vert
a\right\vert \sinh \left( 2\left\vert a\right\vert h+\lambda _{9}\right) %
\right] +\lambda _{10},
\end{gather*}%
for $\lambda _{1},...,\lambda _{10}\in \mathbb{R}.$
\end{enumerate}
\end{theorem}

\begin{proof}
Let $M^{2}$ be locally given by 
\begin{equation*}
(x,z)\longmapsto \mathbf{r}\left( x,z\right) =\left( x,u\left( x,z\right)
,z\right) ,
\end{equation*}%
for a smooth function $u=u\left( x,z\right) $. The normal vector field on $%
M^{2}$ is%
\begin{equation*}
\mathbf{n}=\frac{u_{x}\mathbf{e}_{1}-\mathbf{e}_{2}-u_{z}\mathbf{e}_{3}}{%
\sqrt{1+\left( u_{x}\right) ^{2}-\left( u_{z}\right) ^{2}}}.
\end{equation*}%
Because $M^{2}$ is $\nabla -$singular minimal, we get \autoref{27}. Assume
that $M^{2}$ is $\nabla -$minimal. Due to $\alpha \neq 0,$ \autoref{27} gives $\left\langle \mathbf{n},\mathbf{v}\right\rangle =0$ and 
\begin{equation}
au_{x}+cu_{z}=b.  \label{35}
\end{equation}%
Remark also that we may write $\mathbf{v}=a\mathbf{r}_{x}+c\mathbf{r}_{z},$
implying the tangent plane of $M$ at any point is parallel to $\mathbf{v}.$
The condition of $\nabla -$minimality yields%
\begin{equation}
\left[ \left( u_{z}\right) ^{2}-1\right] u_{xx}-2u_{x}u_{z}u_{xz}+\left[
1+\left( u_{x}\right) ^{2}\right] u_{zz}-2\left[ 1+\left( u_{x}\right)
^{2}-\left( u_{z}\right) ^{2}\right] u_{z}=0.  \label{36}
\end{equation}%
We distinguish several cases:

\begin{enumerate}
\item $a=0,$ $c\neq 0.$ Then $b\neq 0$ because $\mathbf{v}$ is spacelike.
The solution to \autoref{35} is given by $u\left( x,z\right) =\frac{b}{c}%
z+f\left( x\right) ,$ for an arbitrary smooth function $f.$ Hence \autoref{36} turns to
\begin{equation}
\frac{cf^{\prime \prime }}{1-\left( cf^{\prime }\right) ^{2}}=-2b,
\label{37}
\end{equation}%
where $f^{\prime }=\frac{df}{dx},$ etc. Because $M^{2}$ is non-degenerate, $%
1-\left( cf^{\prime }\right) ^{2}\neq 0.$ Therefore the second statement of
the theorem is proved by integrating \autoref{37}.

\item $a\neq 0,$ $c=0.$ Then \autoref{35} gives $u\left( x,z\right) =\frac{b}{a}x+g\left( z\right) $ for an arbitrary smooth function $g$ and so \autoref{36} leads to 
\begin{equation}
g^{\prime \prime }-2\left[ 1-\left( ag^{\prime }\right) ^{2}\right]
g^{\prime }=0,  \label{38}
\end{equation}%
where $g^{\prime }=\frac{dg}{dz},$ etc. That $g^{\prime }=0$ is a trivial
solution to \autoref{38}, implying the first statement of the theorem.
Otherwise, $g^{\prime }\neq 0,$  \autoref{38} may be rewritten as
\begin{equation}
\frac{g^{\prime \prime }}{g^{\prime }}+\frac{a}{2}\left( \frac{g^{\prime
\prime }}{1-ag^{\prime }}-\frac{g^{\prime \prime }}{1+ag^{\prime }}\right)
=2.  \label{39}
\end{equation}%
The third statement of the theorem is obtained by integrating \autoref{39}.

\item $ac\neq 0.$ The solution to \autoref{35} is%
\begin{equation*}
u\left( x,z\right) =\frac{b}{a}x+h\left( z-\frac{c}{a}x\right) ,
\end{equation*}%
for an arbitrary smooth function $h$. Therefore \autoref{36} reduces to
\begin{equation}
h^{\prime \prime }-2\left[ a^{2}+\left( b-ch^{\prime }\right) ^{2}-\left(
ah^{\prime }\right) ^{2}\right] h^{\prime }=0,  \label{40}
\end{equation}%
where $h^{\prime }=\frac{dh}{d\tilde{z}},$ $h^{\prime }=\frac{d^{2}h}{d%
\tilde{z}^{2}},$ $\tilde{z}=z-\frac{c}{a}x.$ We have three subcases: the
first one is that $b=0.$ Then \autoref{40} may be rewritten as
\begin{equation}
\frac{h^{\prime \prime }}{h^{\prime }}+\frac{h^{\prime \prime }}{2\left(
a-h^{\prime }\right) }-\frac{h^{\prime \prime }}{2\left( a+h^{\prime
}\right) }=2a^{2}.  \label{41}
\end{equation}%
Integrating \autoref{41} gives the fourth statement of the theorem. The
second subcase is that $a^{2}=c^{2}$ and $b=\pm 1.$ Then \autoref{40} may be rewritten as
\begin{equation}
\frac{\pm 2ch^{\prime \prime }}{1+c^{2}\mp 2ch^{\prime }}+\frac{h^{\prime
\prime }}{h^{\prime }}=2\left( 1+c^{2}\right) .  \label{42}
\end{equation}%
After integrating \autoref{42}, we obtain the fifth statement of the theorem. The third subcase is that $a^{2}\neq c^{2}.$ Then \autoref{40}
may be rewritten as
\begin{equation}
\frac{-\left( c^{2}-a^{2}\right) h^{\prime \prime }}{a^{2}-\left[ bc-\left(
c^{2}-a^{2}\right) h^{\prime }\right] ^{2}}=2h^{\prime }.  \label{43}
\end{equation}%
A first integration of \autoref{43} yields
\begin{equation}
\frac{\left( c^{2}-a^{2}\right) dh}{-\left\vert a\right\vert \tanh \left(
2\left\vert a\right\vert h+\lambda \right) +bc}=d\tilde{z},  \label{44}
\end{equation}%
for $\lambda \in \mathbb{R}.$ The proof is completed by a first integration of \autoref{44}.
\end{enumerate}
\end{proof}

\begin{remark}\label{rem4}
The last three statements of \autoref{teo4} are new examples in $\mathbb{R}%
_{1}^{3}$ of $\nabla -$minimal surfaces while the second and third
statements are $\nabla -$minimal translation surfaces, introduced by Wang 
\cite{Wang}.
\end{remark}

Let $M^{2}$ be a timelike surface $\mathbb{R}_{1}^{3}$ of type $x=u\left(
y,z\right) $. The spacelike unit normal vector field on $M^{2}$ is%
\begin{equation*}
\mathbf{n}=\frac{\mathbf{e}_{1}-u_{y}\mathbf{e}_{2}+u_{z}\mathbf{e}_{3}}{%
\sqrt{1+\left( u_{y}\right) ^{2}-\left( u_{z}\right) ^{2}}}.
\end{equation*}%
Suppose that $M^{2}$ is $\nabla -$singular minimal with respect to the
vector $\mathbf{v}=\left( a,b,c\right) $. If $M^{2}$ is also $\nabla -$%
minimal, then Eq. \eqref{28} gives%
\begin{equation*}
bu_{y}+cu_{z}=a,
\end{equation*}%
where $b^{2}+c^{2}\neq 0.$ Notice that the mean curvature is same as that of
the surface of type $y=u\left( x,z\right) .$ Therefore, without giving a
proof, we may state a similar result for those surfaces of type $x=u\left(
y,z\right) $ to \autoref{teo4} by replacing $x$ with $y$ and $a$ with $b.$

\subsection{$D-$Singular minimal surfaces}

Consider the following semi-symmetric non-metric connection on $\mathbb{R}%
_{1}^{3}$ \cite{Wang}%
\begin{equation}
D_{\mathbf{x}}\mathbf{y}=\nabla _{\mathbf{x}}^{L}\mathbf{y}+\left\langle 
\mathbf{y},\mathbf{e}_{3}\right\rangle \mathbf{x},  \label{45}
\end{equation}%
where $\mathbf{x},\mathbf{y}$ are tangent vector fields to $\mathbb{R}^{3}.$
The nonzero derivatives are 
\begin{equation*}
D_{\mathbf{e}_{1}}\mathbf{e}_{3}=\mathbf{e}_{1},\text{ }D_{\mathbf{e}_{2}}%
\mathbf{e}_{3}=\mathbf{e}_{2},\text{ }D_{\mathbf{e}_{3}}\mathbf{e}_{3}=%
\mathbf{e}_{3}.
\end{equation*}

\begin{definition}\label{def6}
Let $\mathbf{r}$ be a smooth immersion of an oriented timelike surface $M^{2}
$ in $\mathbb{R}_{1}^{3}$ and $\mathbf{n}$ unit normal vector field on $M^{2}
$ and $H^{D}$ the mean curvature of $M^{2}$ with respect to $D.$ Let $%
\mathbf{v\neq 0}\in \mathbb{R}_{1}^{3}$ a unit fixed spacelike vector
non-parallel to $\mathbf{n}$ such that $\mathbf{n}$ and $\mathbf{v}$ span a
spacelike 2-space. $M^{2}$ is called $D-$\textit{singular minimal surface}
with respect to $\mathbf{v}$ if satisfies%
\begin{equation}
2H^{D}=\alpha \frac{\left\langle \mathbf{n},\mathbf{v}\right\rangle _{L}}{%
\left\langle \mathbf{r},\mathbf{v}\right\rangle _{L}},\text{ }\alpha \in 
\mathbb{R},\text{ }\alpha \neq 0.  \label{46}
\end{equation}
\end{definition}

The surface $M^{2}$ is called $D-$\textit{minimal} if $H^{D}=0$. With \autoref{def6}, we first observe the $D-$singular minimal surfaces of type $%
z=u\left( x,y\right) $ which are $D-$minimal. Hence \autoref{46} gives $
\left\langle \mathbf{n},\mathbf{v}\right\rangle =0$ and 
\begin{equation}
au_{x}+bu_{y}=c,  \label{47}
\end{equation}%
where $\mathbf{v}=\left( a,b,c\right) $. The condition of $D-$minimality
yields%
\begin{equation}
\left[ 1-\left( u_{y}\right) ^{2}\right] u_{xx}+2u_{x}u_{y}u_{xy}+\left[
1-\left( u_{x}\right) ^{2}\right] u_{yy}=0,  \label{48}
\end{equation}%
where the roles of $x$ and $y$ are symmetric. If $a=0$, then \autoref{47}
follows $u\left( x,y\right) =f\left( x\right) +\frac{c}{b}y,$ for an
arbitrary smooth function $f.$ Considering this into \autoref{48} yields $\frac{1}{b^{2}}\frac{d^{2}f}{dx^{2}}=0,$ which leads $M^{2}$ to be a plane
parallel to $\mathbf{v}.$ By symmetry, we can obtain same result when $a\neq
0$ and $b=0$. Let $ab\neq 0.$ Then the solution to \autoref{47} is $u\left( x,y\right) =\frac{c}{a}x+g\left( y-\frac{b}{a}x\right) ,$ for a
smooth function $g.$ After substituting its partial derivatives into \autoref{48}, we conclude $\frac{1}{a^{2}}\frac{d^{2}g}{d\tilde{y}^{2}}=0,$ $%
\tilde{y}=y-\frac{c}{a}x,$ yielding that $M$ is a plane parallel to $\mathbf{v}.$

Therefore we state the following

\begin{theorem}\label{teo6}
Let $M^{2}$ be a $D-$singular minimal surface in $\mathbb{R}_{1}^{3}$ of
type $z=u\left( x,y\right) $ with respect to a unit spacelike vector $%
\mathbf{v}.$ If $M^{2}$ is $D-$minimal, then it is a plane parallel to $%
\mathbf{v}.$
\end{theorem}

When we take the surfaces of type $y=u\left( x,z\right) $ or $x=u\left(
y,z\right) ,$ we may state a similar result to \autoref{teo6}.

\section{Conclusions}

In this study, we discussed the singular minimal surfaces in $\mathbb{R}^{3}$
(resp. $\mathbb{R}_{1}^{3}$) which are minimal and expressed a trivial
outcome, \autoref{prop1} (resp. \autoref{prop2}). Nevertheless, the non-trivial
outcomes, Theorems \ref{teo1} and \ref{teo2} (resp. Theorems \ref{teo4} and \ref{teo5}), were obtained by using
the modified version, \autoref{def3} (resp. \autoref{def5}), of singular
minimality. With this definition, we observed that the singular minimal
surfaces which are minimal are the generalized cylinders and thus our
approach\ also may be viewed as an effort of finding $\nabla -$minimal
generalized cylinders.

The generalized cylinders belong to a subcase of translation surfaces. For
this reason, the $\nabla -$minimal translation surfaces introduced by Wang 
\cite{Wang} were presented by some of our results. Still, we also exhibited
new examples of $\nabla -$minimal surfaces, as explained in Remarks \ref{rem1} and \ref{rem2}
(resp. Remarks \ref{rem3} and \ref{rem4}). Morever, a trivial outcome, \autoref{teo3} (resp.
\autoref{teo6}), was found by using the semi-symmetric non-metric connection $D$
given by \autoref{23} (resp. \autoref{45}).

\end{document}